\documentclass[11pt]{amsart}

\usepackage{amssymb}

\usepackage[bookmarks,colorlinks,citecolor=red,linkcolor=blue]{hyperref}
\usepackage{tikz}

\newtheorem{theorem}{Theorem}[section]
\newtheorem{lemma}[theorem]{Lemma}
\newtheorem{proposition}[theorem]{Proposition}
\newtheorem{corollary}[theorem]{Corollary}

\theoremstyle{definition}

\newtheorem{definition}[theorem]{Definition}

\theoremstyle{remark}
\newtheorem{remark}[theorem]{Remark}

\newcommand{\haus}{\mathcal{H}^1}

\title[Ahlfors regularity of continuum]{Ahlfors regularity of continua that minimize maxitive set functions}
\author{Davide Zucco}
\address{Dipartimento di Matematica ``G. Peano'', Università di Torino, Via Carlo Alberto 10, 10123 Torino, Italy}
\email{davide.zucco@unito.it}

\begin{document}

\begin{abstract}
The primary objective of this paper is to establish the \emph{Ahlfors regularity} of minimizers of set functions that satisfy a suitable \emph{maxitive} condition on disjoint unions of sets. Our analysis focuses on minimizers within continua of the plane with finite 1-dimensional Hausdorff measure. Through quantitative estimates, we prove that the length of a minimizer inside the ball centered at one of its points is comparable to the radius of the ball. 

By operating within an abstract framework, we are able to encompass a diverse range of entities, including the \emph{inradius of a set}, the \emph{maximum of the torsion function}, and \emph{spectral functionals} defined in terms of the eigenvalues of elliptic operators. These entities are of interest for several applications, such as structural engineering, urban planning, and quantum mechanics.
\end{abstract}

\maketitle
{\small


}

\section{Introduction}

A set function $F$ is \emph{maxitive} on a suitable class of sets if 
on every disjoint union of sets $A_1$ and $A_2$ of this class there holds
\begin{equation}\label{global}
F(A_1\cup A_2)=\max \{F(A_1),F(A_2)\}.
\end{equation}
This maxitivity condition is not uncommon in the literature and is notable for its application in Probability Theory as an alternative to the additivity property of a measure. This has led to the emergence of a new and significant field of research known as Possibility Theory (for an overview on this subject, see, for instance, \cite{dubpra} and \cite{pon}). In this context, $F$ aims to measure \emph{the possibility} of an event and thus $F$ is required to be \emph{maxitive on a $\sigma$-algebra of sets}, a typical class used for modeling the events. 

However, in this paper, we address a set function $F$ defined solely on \emph{open sets} (thus within a space that is not closed under complement) and further, which is \emph{maxitive on open sets}. This means that the value of $F$ for each open set equals the maximum of the values over its connected components. Our approach is abstract, allowing us to embrace a diverse array of intriguing set functions. Indeed within this framework, we can incorporate various classical entities, such as the \emph{inradius} $R$ of an open set $A$ 
$$R(A):=\max_{x\in A} \mathrm{d}(x,\partial A)$$
\emph{the maximum $M$ of the torsion function} $w_A$
\[
M(A):= \max_{x\in A} w_A(x),
\]
and \emph{spectral functionals} given by the eigenvalues $\lambda_j$ of an elliptic operator
\[
f\big(\lambda_1(A),\lambda_2(A),\dots, \lambda_k(A)\big),
\]
with $f\colon \mathbb R^{k}\to \mathbb R$ decreasing in each variable. These entities are of interest for several applications, such as structural engineering, urban planning, and quantum mechanics.

In such applications, primarily grounded in minimum principles, it is sometimes advantageous to seek the best open sets $A$ that minimize $F$. Without additional constraints, the minimization problem becomes trivial because the open set $A$ can simply be taken as $\emptyset$: indeed due to the monotonicity of $F$ with respect to set inclusion (recall that, like additive measures, maxitive set functions are monotone), $\emptyset$ is an obvious minimizer for $F$. Therefore, to ensure the existence of non-trivial solutions, it is necessary to prevent the convergence of minimizing sequences to $\emptyset$. A common approach, often motivated by the applications (see, e.g., \cite{buoust,chetep, DMT, mipast, paoste, tilzuc, tilzuc2, zuc}) is to work with functions defined solely on open sets of the form $\Omega\setminus\Sigma$ for some fixed open set $\Omega$ and some closed set $\Sigma$. Specifically, one first fixes $\Omega$ and then studies the set function $F$ by varying \emph{closed sets} within a suitable class that prevents the convergence of a sequence of $\Sigma$ to $\overline\Omega$ (which would prevent the convergence of the corresponding open sequence $\Omega\setminus \Sigma$ to the empty set). In two dimensions, this can be achieved by imposing a uniform bound on the number of connected components of $\Sigma$ together with an upper bound on its length (see for instance \cite[Corollary~3.3]{DMT}). For simplicity in presentation, we only focus on \emph{connected} sets in the plane (not on multi-connected ones) and therefore, the prototype of an admissible set $\Sigma$ must be a curve or a connected system of curves, with a given total length. Other generalizations are clearly possible, as one may also consider curve-type $\Sigma$ in higher-dimensions (but we do not pursue this technicality in this paper).

For a given open bounded connected set $\Omega\subset \mathbb R^2$ and a real number $L>0$ we focus on existence and regularity results for the following shape optimization problem: 
\begin{equation}\label{prob}
\min\{ F(\Omega\setminus \Sigma): \text{$\Sigma\subset \overline \Omega$ closed, connected, and $\haus(\Sigma)=L$}\},
\end{equation}
where $F\colon \mathcal A(\Omega)\to [0,+\infty]$ is a set function defined on the open subset of $\Omega$. The existence of a minimizer is quite standard and follows the direct methods of the Calculus of Variations. One must use the Hausdorff topology (that we recall in Section~\ref{sec.sets}) and combine the (very natural) hypothesis that $F$ is lower-semicontinuous with that it is monotone (see Section~\ref{sec.func} for definitions). 

\begin{theorem}[Existence]\label{teo:ex}
Let $F\colon \mathcal A(\Omega)\to [0,+\infty]$ be a set function that is 
\begin{itemize}
\item[(i)] lower semicontinuous w.r.to the Hausdorff convergence of open sets;
\item[(ii)] monotone w.r.to set inclusion;
\end{itemize}
i.e., $F$ satisfies \eqref{eq.ls} and \eqref{eq.mo}.
Then there exists a minimizer for problem \eqref{prob}.
\end{theorem}

The regularity of a minimizer (in this quite general framework) is instead less obvious and is the principal target of this study. 
Under suitable assumptions on $F$ (that we will discuss in a moment) one can prove the so-called \emph{Ahlfors regularity}, a weak notion of regularity, which is usually defined in terms of two density estimates through the ball $B_r(x)$ with center $x$ and radius $r$: precisely it states that the total length of a minimizer inside the ball centered at one of its points is comparable to the radius of the ball. 
\begin{definition}[Ahlfors regularity]\label{alfhors}
A set $\Sigma\subset \mathbb R^2$ is said \emph{Ahlfors regular} if there exist two positive constants $c_1,c_2>0$ and a radius $r_0>0$ such that for every positive $r<r_0$ there hold the estimates
\begin{equation}\label{ahl}
c_1r\leq \haus(\Sigma\cap B_r(x))\leq c_2r, \qquad \text{for every $x\in\Sigma$},
\end{equation} 
where $B_r(x)$ denotes the ball with center $x$ and radius $r$.
\end{definition}
In other words a set is Ahlfors regular if its contribution (in terms of the $1$-dimensional Hausdorff measure) inside a ball is uniformly of the same order of the radius of the ball. Quantitatively this says that the total length of the set inside the ball centered at one of its points of radius $r$ is no more than $c_2$ times the radius of the ball and no less than $c_1$ times the radius of the ball. When the set is \emph{connected} (e.g., it is a continuous curve) the lower bound is always satisfied with $c_1=1$ (see \cite[Theorem~4.4.5]{ambtil}) (equality is achieved when $\Sigma\cap B_r(x)$ is a line segment with one endpoint at $x$). Moreover, when the set is \emph{regular} in a classical sense (e.g., it is a simple Lipschitz curve), the upper bound is also satisfied with $c_2$ equals approximately $2$. At a $k$-node of a curve (i.e., a point where the curve exits with $k$ regular branches), the upper bound remains valid but with a constant $c_2\approx k$. Essentially, the constant $c_2$ accounts for the maximum number $k$ of $k$-nodes of a curve. An example of a set that is not Ahlfors regular at one of its point is the following. For $k\in\mathbb N$ let $s_k$ (in polar coordinate) be the radius of length $1/(2^k)$ of angle $\pi/(k+1)$ of the ball centered at the origin. The countable union $S$ of all these radii $s_k$ is a continuum that is rectifiable but not Ahlfors regular, see Figure~\ref{fig.alf}. Indeed $\haus(S)=2$ yields the rectifiability of $S$. Moreover, since the function $h(r):=\haus(S\cap B_r(0))$ is decreasing, the limit of $h$ as $r\to 0^+$ exists and it suffices to compute it along a subsequence. If $r=1/2^{n}$ then $h(r)=(2+n)/2^n=(2+n)r$, and this excludes the existence of a constant $c_2$ bounding $(2+n)$ for all $n\in\mathbb N$. A more refined rectifiable continuum that is not Ahlfors regular in any of its points has been recently constructed in 
\cite[Theorem~1.1]{gomc} by unions of approximations of snowflake-like sets. 

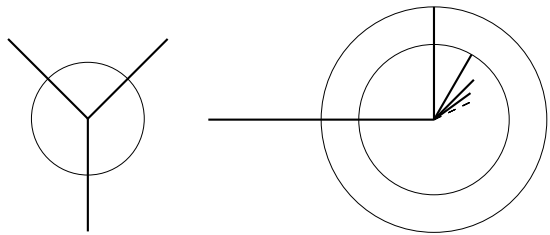
\begin{figure}
\begin{center}
\begin{tikzpicture}[scale=1.5]

    \fill (0,0) circle (0.01);

    \draw[-,thick] (0,0) -- ({cos(270)},{sin(270)});
    \draw[-, thick] (0,0) -- ({cos(135)},{sin(135)});
    \draw[-, thick] (0,0) -- ({cos(45)},{sin(45)});
    \draw[very thin] (0,0) circle (0.5);

\end{tikzpicture}
\quad
\begin{tikzpicture}[scale=3]

    \draw[very thin] (0,0) circle (1/2);
    \draw[very thin] (0,0) circle (1/3);
    \foreach \k in {1,2,3,4,5} {
        \draw[-, thick] (0,0) -- ({cos(180/\k)*1/\k}, {sin(180/\k)*1/\k});
        \draw[dashed, -] (0,0) -- ({cos(180/7)*1/5}, {sin(180/7)*1/5});
    }
\end{tikzpicture}
\caption{Two rectifiable continua, the left is Ahlfors regular, the right is not.}\label{fig.alf}
\end{center}
\end{figure}

Despite its simplicity, Ahlfors regularity is already a first step towards achieving more regularity results; note for instance that it implies the so-called uniform rectifiability, the uniform concentration property, some type of Carleson-measure estimates, (see   \cite[Chapter C]{dav} and \cite{davsem} for further details). Moreover, Ahlfors regularity, together with a blow-up analysis, could be used to derive information on the local geometry of the minimizers (in the spirit of \cite{santil}). We plan to investigate this topics in a future work.

To state the main result of the paper the monotonicity condition in Theorem~\ref{teo:ex} has to be refined in two directions. Firstly, one needs $F$ to be \emph{strictly monotone} as all connected components of a set shrink. Moreover, one 
requires $F$ to be \emph{locally} maxitive, which roughly consists in supposing that the value of $F$ along the union of a set with a small ball equals the values of the set (see Section~\ref{sec.func} for precise definitions). 
\begin{theorem}[Regularity]\label{teo:reg}
Let \!$F\colon\! \mathcal A(\Omega)\to\! [0,+\infty]$ be a set function that is 
\begin{itemize}
\item[(i)] lower semicontinuous w.r.to Hausdorff convergence of open sets;
\item[(ii)] strictly monotone on enlargements; 
\item[(iii)] locally maxitive;
\end{itemize}
i.e., $F$ satisfies conditions \eqref{eq.ls}, \eqref{eq.sm}, and \eqref{eq.wm}.
Then every minimizer $\Sigma$ of problem \eqref{prob} is Ahlfors regular: there exists a radius $r_0>0$, such that for every $r<r_0$ it holds
\begin{equation}\label{astima}
r\leq\haus(\Sigma\cap B_r(x))\leq 2\pi r, \qquad \text{for any $x\in \Sigma$},
\end{equation}
where $B_r(x)$ denotes the ball with center $x$ and radius $r$.
\end{theorem}

In Section~\ref{sec.prob} we prove Theorems~\ref{teo:ex} and \ref{teo:reg}. The concept of the proof of the regularity result is straightforward and relies on a comparison between a set and a new one, defined by replacing the portion of the initial set inside the ball with the boundary of the ball. This primarily builds upon an idea already present in \cite{mipast} for the so-called \emph{maximum distance problem}, that have been deeply studied by several authors (see e.g. \cite{chetep, paoste, tilzuc}). Of course, we can derive the same result in \cite[Theorem~5.1]{mipast} from the previous Theorem~\ref{teo:reg}.

\begin{corollary}[Maximum distance problem]\label{cor.mdp}
There exists a minimizer of
\begin{equation}\label{mdp}
\min\{R(\Omega\setminus \Sigma): \text{$\Sigma\subset \overline \Omega$ closed, connected, and $\haus(\Sigma)=L$}\}
\end{equation}
and every minimizer is Ahlfors regular satisfying \eqref{astima}.
\end{corollary}

A potential interpretation of problem \eqref{mdp} can be framed in terms of urban planning: in a city (modeled by the set $\Omega$), there may be a need to design a new infrastructure, such as a metro line, an irrigation system, or an internet network (modeled by the set $\Sigma$). The goal is to optimize the layout of this infrastructure $\Sigma$ within $\Omega$, taking into account constraints on its length imposed by cost consideration and practicality of expansion.

The novelty of Theorem~\ref{teo:reg} lies in our generalization to an abstract framework, allowing us to encompass a diverse range of entities unified by the shared criterion of satisfying a suitable \emph{strict monotonicity} condition with a \emph{local maxitivity} property. The interest on these two conditions, shrinking the size of \emph{all} connected components, and not of merely one of them, as well as being locally maxitive, in view of globally maxitive \eqref{global}, is motivated by later applications of the main result Theorem~\ref{teo:reg} to some specific set functions (such as those depending on the eigenvalues of the Dirichlet-Laplacian, which in fact only verify the local version of the maxitivity condition).

\begin{corollary}[Spectral optimization]\label{spectra}
Let $k\in\mathbb N$ and $f\colon \mathbb R^{k}\to \mathbb R$ be a \emph{lower semicontinuous} function that is \emph{strictly decreasing in each variable}, i.e., if $x_j\leq y_j$ for all $j\in\{1,\dots,k\}$ and $x_{j_0}<y_{j_0}$ for at least one $j_0\in\{1,\dots,k\}$ then $f(x_1,x_2,\dots, x_k)> f(y_1,y_2,\dots, y_k)$.
There exists a minimizer of 
\begin{equation*}
\min\{f\big(\lambda_1(\Omega\setminus\Sigma),\dots, \lambda_k(\Omega\setminus\Sigma)\big): \text{$\Sigma\subset \overline \Omega$ closed, connected, and $\haus(\Sigma)=L$}\}
\end{equation*}
and every minimizer is Ahlfors regular satisfying \eqref{astima}.
\end{corollary}
By taking $f(\lambda_1,\dots, \lambda_k)=1/\lambda_k$ in Corollary~\ref{spectra} allows to formulate the minimization problem as a maximization one and to obtain existence and Ahlfors regularity of maximizers of the following problem
\begin{equation}\label{CH}
\max\{ \lambda_k(\Omega\setminus \Sigma): \text{$\Sigma\subset \overline \Omega$ closed, connected, and $\haus(\Sigma)=L$}\}.
\end{equation}
The case $k=1$ has been extensively studied in \cite{tilzuc, tilzuc2}, where  the asymptotic behavior of maximizers $\Sigma_L$ as the length $L\to+\infty$ was derived. A one dimensional version of this problem, focusing on optimizing points instead of curves, was posed by Courant and Hilbert in \cite[pp 463--464]{couhil} and analyzed in considerable generality in \cite{tilzuc3}. These investigations find motivation, for instance, in structural engineering applications, particularly in the pursuit of more resilient structure (defined as structures with higher frequencies of vibrations). Such stiffening can be achieved through reinforcement (e.g., attachment of some rigid and hard material) and by fixing the structure along additional areas, modeled by the set $\Sigma$ (where the parameter $L$ can be interpreted as a constraint on construction costs). 

We refer to Section~\ref{sec.app} for the proofs of Corollaries~\ref{cor.mdp} and \ref{spectra} as well as for further applications, including \emph{Poincar\'e-Sobolev constants} and \emph{the maximum of the torsion function} (which now are maxitive).  It is worth mentioning that \emph{non-maxitive} set functions do not fall within this framework. A distinguished example is the torsional rigidity (that is the $L^1$-norm of the torsion function), which is additive (and thus not maxitive). A specific regularity result for this set function has been however obtained in \cite[Section 4.2]{clls}.

\section{Hausdorff topology for closed and open sets}\label{sec.sets}

Let $\Omega$ be a fixed open, bounded, and connected subset of $\mathbb R^2$, $\mathcal A(\Omega)$ and $\mathcal C(\overline\Omega)$ denote the classes of open and closed subsets of $\Omega$ and of $\overline \Omega$, respectively. 
For a closed set $\Sigma$ in $\mathcal C(\overline \Omega)$ we define the \emph{distance function} to $\Sigma$ as 
\begin{equation}\label{eq:dist}
\mathrm{d}(x,\Sigma):=\min_{y\in \Sigma} |x-y|,
\end{equation}
 where $|\cdot |$ is the Euclidean distance in $\mathbb R^n$. The minimum is achieved by Weierstrass theorem, thanks to the compactness of $\Sigma$ and the continuity of $|\cdot |$. Given $\Sigma_1$ and $\Sigma_2$ closed sets in $\mathcal C(\overline\Omega)$ we set $$\rho(\Sigma_1,\Sigma_2):=\max_{x\in \Sigma_1} \mathrm d(x,\Sigma_2)$$ and define the \emph{Hausdorff distance} $\mathrm{d_H}$ between the closed sets $\Sigma_1$ and $\Sigma_2$ as 
$$\mathrm{d_H}(\Sigma_1,\Sigma_2):=\max\{\rho(\Sigma_1,\Sigma_2),\rho(\Sigma_2,\Sigma_1)\}.$$
Equivalently, in terms of functions, one has
\begin{equation}\label{supdist}
\mathrm{d_H}(\Sigma_1,\Sigma_2)=\|\mathrm d(\cdot,\Sigma_1)-\mathrm d(\cdot,\Sigma_2)\|_\infty.
\end{equation}
Then we have the following.
\begin{definition}[Hausdorff convergence of closed sets]
A sequence $\{\Sigma_n\}_n$ of closed sets of $\mathcal C(\overline \Omega)$ is said \emph{Hausdorff converging in $\mathcal C(\overline \Omega)$} if there exists a closed set $\Sigma$ of $\mathcal C(\overline \Omega)$ such that $\mathrm{d_H}(\Sigma_n,\Sigma)\to0$ as $n\to\infty$.
\end{definition}
The family of closed sets endowed with the Hausdorff convergence has two fundamental properties. The first one concerns compactness.

\begin{theorem}[Blaschke]\label{teo.b}
Up to subsequences, every sequence of closed sets $\{\Sigma_n\}_n$ of $ \mathcal C(\overline \Omega)$ is Hausdorff converging in $\mathcal C(\overline \Omega)$.
\end{theorem}

The second result concerns the lower semicontinuity of the 1-dimensional Hausdorff measure with respect to the Hausdorff convergence of closed and connected sequences of sets. Connectedness is sufficient and not too far from being necessary to achieve this result. A necessary condition (and sufficient, see \cite[Corollary~3.3]{DMT}) would be a uniform bound on the number of connected components. Indeed if $\Omega$ is the unit square and $\Sigma_n$ is the union of the $n$ points $\cup_{j=1}^n\{(j/n,0)\}$ then $\Sigma_n$ is a sequence of closed sets, Hausdorff converging in $\mathcal C(\overline \Omega)$ to $\Sigma=[0,1]\times \{0\}$ and such that $\haus(\Sigma)>\haus(\Sigma_n)$ for all $n\in\mathbb N$, contradicting the lower semicontinuity of $\haus$. For simplicity we will only deal with \emph{continua}, i.e., non-empty connected closed sets. Another advantage is that connectedness passes to the limit in Hausdorff convergence.

\begin{theorem}[Go\l ab]\label{teo.g}
Let $\{\Sigma_n\}_n$ be a sequence of continua of $\mathcal C(\overline \Omega)$ Hausdorff converging in $\mathcal C(\overline \Omega)$ to some closed set $\Sigma$. Then $\Sigma$ is a continuum and
\[
\haus(\Sigma)\leq \liminf_{n\to\infty} \haus(\Sigma_n).
\]
\end{theorem}

There are several references where the proofs are provided the interested reader can consult, for instance, \cite[Theorem~4.4.15]{ambtil} and \cite[Theorem~4.4.17]{ambtil}  which contain proofs in the general framework of metric spaces. 

For introducing the strict monotonicity condition, we need the following.

\begin{definition}[Enlargement of continua]
Given $\Sigma\in\mathcal{C}(\overline\Omega)$ connected, an \emph{enlargement} $\Sigma'\in\mathcal{C}(\overline\Omega)$ of $\Sigma$ is a set that satisfies the following properties:
\begin{itemize}
\item (\emph{Connectedness}) $\Sigma'$ is connected;
\item (\emph{Containment}) $\Sigma'\supsetneq \Sigma$;
\item (\emph{Shrinking of connected components}) for each connected component $A'$ of $\Omega\setminus \Sigma'$ there exists a connected component $A$ of $\Omega\setminus \Sigma$ with $A'\subsetneq A$.
\end{itemize}
\end{definition}

The interest of shrinking the size of \emph{each} connected component of $\Omega\setminus \Sigma$, and not of merely one of them, is motivated by later applications of the main result Theorem~\ref{teo:reg} to some specific problems. Indeed, the examples of set functions used in this paper are decreasing only when all the connected components shrink. The following result shows (in a constructive way) that there always exists such an enlargement.

\begin{lemma}\label{lemma2}
Let $L>0$ and $\Sigma\in\mathcal{C}(\overline\Omega)$ be a connected set with $0<\haus(\Sigma)< L$. Then there exists an enlargement $\Sigma'\in\mathcal{C}(\overline\Omega)$ of $\Sigma$ with $\haus(\Sigma')=L$.
\end{lemma}
\begin{proof}
We divide the construction of an enlargement $\Sigma'$ of $\Sigma$ in two steps. Then, from this construction, it turns out that $\Sigma'$ has length $L$.
 
1. \emph{Shrinking a connected component of $\Omega\setminus \Sigma$.} Consider a connected component $A$ of $\Omega\setminus\Sigma$. Since $\haus(\Sigma)< L$ and $\Sigma$ is closed then there exists a point $x_0\in A$ with $\mathrm d(x_0,\Sigma)>0$.  
By definition, the minimum in \eqref{eq:dist} is taken at least for one point $y_0\in \Sigma$, thus $y_0\neq x_0$ and $\mathrm d(x_0,\Sigma)=|x_0-y_0|>0$. Then consider $x_t=(1 - t)x_0 + ty_0$ for some large enough $t\in [0,1)$ so that the segment $s$ linking $x_t$ to $y_0$ lies inside $A$ (except of course at the endpoint $y_0$).  
Notice that, by minimality, $y_0$ is the unique point of the segment $s$ that also belongs to $\Sigma$ and $y_0$ is then still such that $$\mathrm d(x_t, \Sigma) = |x_t-y_0|= |(1 - t)x_0 + t y_0 - y_0|=(1-t)|x_0-y_0|>0.$$
Since $A$ is an open set we may consider a ball $B_r(x_t)$ with center $x_t\in A$ and radius $r>0$ small enough such that $B_{r}(x_t)\subset A$;  
as $B_r(x_t)\subset A\subset\Omega\setminus \Sigma$, then necessarily $B_r(x_t)$ does not intersect $\Sigma$. To extend $\Sigma$ to an appropriate length, we can add to $\Sigma$ the line $s$ (or a portion thereof) and any other radius (or a portion of radius) leaving the center $x_t$ of $B_r(x_t)$, as showed in Figure~\ref{figura}.
 
2. \emph{Shrinking all connected component of $\Omega\setminus \Sigma$.} We apply the previous construction to each connected component $A_i$ of $\Omega\setminus\Sigma$ (since $\Omega\setminus \Sigma$ is an open set of $\mathbb R^2$, these components are at most countable). Then, we obtain at most countably many $\{s_i\}_i$ of arbitrary lengths and we can select these lengths in such a way that the resulting set $\Sigma'$ matches the length $L$.
 \end{proof}

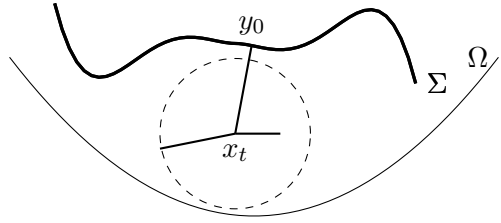
\begin{figure}
\centering
\begin{tikzpicture}[>=latex,scale=0.2, domain=-4:4, samples=50]
 
  \coordinate (x0) at (0,5) {}; 
  \coordinate (y0) at (1.1,10.9) {}; 
  \coordinate (x1) at (3,5) {}; 
  \coordinate (x2) at (-5,4) {}; 
 
  \draw[thick, -] (x0) node[below]  {$x_t$} -- (y0) node[above] {$y_0$}; 
  \draw[thick, -] (x0) -- (x1) ;
  \draw[thick, -] (x0) -- (x2); 
  \draw[dashed] (x0) circle (5cm);
  \draw[very thick, color=black] plot (3*\x,{-1/4*\x^4+4/3*\x^3-3/2*\x^2+11}) node[right] {$\Sigma$};
  \draw[very thick, color=black] plot (3*\x,{-1/4*\x^4+4/3*\x^3-3/2*\x^2+11}) node[right] {$\Sigma$};
  \draw[domain=-3:3.5] plot (5*\x,{(\x-1/2)^2+1/2*\x-2/3}) node[left] {$\Omega$};

\end{tikzpicture}
 \caption{How to extend the length of $\Sigma$.} \label{figura}
\end{figure}

The Hausdorff convergence of open sets is simply defined by taking complements, employing the definition provided for closed sets.
\begin{definition}[Hausdorff convergence of open sets]
A sequence $\{A_n\}_n$ of open sets of $\mathcal A(\Omega)$ is said \emph{Hausdorff converging in $\mathcal A(\Omega)$} if the sequence $\{\overline\Omega\setminus A_n\}_n$ of closed sets of $\mathcal C(\overline \Omega)$ is Hausdorff converging in $\mathcal C(\overline\Omega)$.
\end{definition}  

The family of open sets in $\Omega$ endowed with the Hausdorff convergence is also compact (as already noticed in \cite[Corollary 2.2.26]{HP} this follows rather directly from Theorem \ref{teo.b} reasoning with the complements).

\begin{corollary}[Blaschke]\label{cor.bla}
Up to subsequences, every sequence of open sets $\{A_n\}_n$ of $ \mathcal A(\Omega)$ is Hausdorff converging in $\mathcal A(\Omega)$.
\end{corollary}

Keep in mind that these two notions of Hausdorff convergence (for closed and open sets) are quite distinct: results that can be true for closed sets may not apply to open sets, and vice versa. A typical example is connectedness, which does not necessarily pass to the limit for sequences of open sets, see \cite[page 35]{HP} (unlike the limit of connected closed sets, as stated earlier in the Golab theorem). 

We will also need the behavior of Hausdorff convergence with respect to monotone (with respect to set inclusion) sequences of sets, for instance see \cite[Paragraph 2.2.3.2]{HP}. 

\begin{lemma}\label{lem.mon}
A nondecreasing (in the sense of set inclusion) sequence of open sets of $\mathcal A(\Omega)$ is Hausdorff converging in $\mathcal A(\Omega)$  to its union.
\end{lemma}

\section{Set functions}\label{sec.func}

We consider the set function $F\colon \mathcal A(\Omega)\to [0,+\infty]$,
defined merely on the open subsets of $\Omega$ and we allow for possibly unbounded values.
We introduce several definitions, motivated by set functions for which we apply these results, as detailed in Section~\ref{sec.app}. 

The first two definitions are necessary for the existence of the minimization problem; the others are specific for regularity results. 

\begin{definition}[Lower semicontinuity and continuity]
A given set function $F\colon \mathcal A(\Omega)\to [0,+\infty]$ is \emph{lower semicontinuous} if whenever $\{A_n\}_n$ is Hausdorff converging in $\mathcal A(\Omega)$ to $A$
\begin{equation}\label{eq.ls}
\liminf_{n\to\infty} F(A_n)\geq F(A).
\end{equation}
We call $F$ \emph{continuous}, whenever this $\liminf$ is a limit and the inequality above holds with the equality.
\end{definition}

\begin{definition}[Monotonicity]
A set function $F\colon \mathcal A(\Omega)\to [0,+\infty]$ is \emph{monotone} if whenever $A_1,A_2\in\mathcal A(\Omega)$ with $A_1\subset A_2$
\begin{equation}\label{eq.mo}
F(A_1)\leq F(A_2).
\end{equation}
\end{definition}

The Lebesgue measure is a typical example of lower semicontinuous and monotone set function, see \cite[Proposition 2.2.23]{HP}. Now, for the regularity result, we need refined versions of the monotonicity property.

\begin{definition}[Strict monotonicity] A set function $F\colon \mathcal A(\Omega)\to [0,+\infty]$ is \emph{strictly monotone on enlargements} if it is monotone and whenever $\Sigma\in\mathcal C(\Omega)$ is connected with $0<\haus(\Sigma)< L$ for some $L>0$ there exists an enlargement $\Sigma'$ of $\Sigma$ such that
\begin{equation}\label{eq.sm}
F(\Omega\setminus \Sigma')< F(\Omega\setminus\Sigma).
\end{equation}
\end{definition}

Lemma~\ref{lemma2} ensures that an enlargement of a set can always be found. However, notice that the strict monotonicity might hold for some other, possibly unknown, enlargement.  We also point out that the strict monotonicity is only required to hold on \emph{at least one} enlargement. This is because many set functions are strictly monotone only on some enlargement, rather than on all possible enlargements (this is the case, for instance, for the inradius of a set).

To introduce the next notion we take a detour through an important family of set functions. A noteworthy subclass of monotone set functions is represented by maxitive set functions, which have found application across various research domains as a viable alternative to additive functions.

\begin{definition}[Maxitivity]
A set function $F\colon \mathcal A(\Omega)\to [0,+\infty]$ is \emph{maxitive} if whenever $A_1,A_2\in\mathcal A(\Omega)$ with $A_1\cap A_2=\emptyset$ 
\begin{equation}\label{eq.ma}
F(A_1\cup A_2)= \max \{F(A_1),F(A_2)\}.
\end{equation}
\end{definition}

The following results, Proposition~\ref{prop.mon}, Proposition~\ref{prop.mon2}, and Corollary~\ref{cor.mon3}, concerning some properties of maxitive set functions, have their own interest, but they do not serve for obtaining the main results of the paper. 

It is immediate to prove that a maxitive set function is also monotone. In fact, as shown in the next result, monotone set functions are equivalent to set functions that satisfy a supermaxitivity condition. 

\begin{proposition}[Monotonicity\&Maxitivity]\label{prop.mon}
A set function $F\colon \mathcal A(\Omega)\to [0,+\infty]$ is monotone if and only if whenever $A_1,A_2\in\mathcal A(\Omega)$ with $A_1\cap A_2=\emptyset$
\[
F(A_1\cup A_2)\geq \max \{F(A_1),F(A_2)\}.
\]
Then, a maxitive set function $F$ is always a monotone set function.
\end{proposition}
\begin{proof}
If $F$ is monotone then by trivial inclusions $F(A_1\cup A_2)$ is greater either of $F(A_1)$ and of $F(A_2)$, thus $F$ is supermaxitive. Conversely, if $F$ is a supermaxitive function and $A_1\subset A_2$ then $A_2=A_1\cup (A_2\setminus A_1)$ and we have
\[
F(A_1)\leq \max \{F(A_1),F(A_2\setminus A_1)\}\leq F(A_1\cup (A_2\setminus A_1))=F(A_2),
\]
which proves the monotonicity of $F$.
\end{proof}

\begin{remark}
\emph{Additive} set functions fulfill the supermaxitivity condition stated in Proposition~\ref{prop.mon}, thereby represent another subclass of monotone functions. Prominent examples of additive set functions are measures. 
\end{remark}
 
By combining the maxitivity of a set function with its lower semicontinuity one may infer a $\sigma$-maxitivity property, which holds more generally over the union of countable connected components.
\begin{proposition}[$\sigma$-maxitivity] \label{prop.mon2}
Let $F\colon \mathcal A(\Omega)\to [0,+\infty]$ be a lower semicontinuous and maxitive set function. 
If $A_n\in\mathcal A(\Omega)$ for every $n\in \mathbb N$ with $A_i\cap A_j=\emptyset$ for every $i,j\in\mathbb N$ such that $i\neq j$ then
\[
F\left(\bigcup_{n=1}^\infty A_n\right)=\sup_{n\in\mathbb N} \{F(A_n)\}.
\]
\end{proposition}
\begin{proof}
We start by proving the inequality $\geq$. 
By Proposition~\ref{prop.mon} the set function $F$ is monotone, thus since for $i\in\mathbb N$ the set $A_i\subset \bigcup_{n=1}^\infty A_n$ we have 
\[
F\left(\bigcup_{n=1}^\infty A_n\right)\geq F(A_i).
\]
By the arbitrariness of $i\in\mathbb N$ we obtain the desired inequality. 

To prove the reverse inequality $\leq$ consider $N\in\mathbb N$ and apply $(N-1)$ times the maxitivity condition to obtain
\[
F\left(\bigcup_{n=1}^N A_n\right)=\max_{1\leq n\leq N} F(A_n)\leq \sup_{n\in\mathbb N} F(A_n).
\]
Now, for every $N\in\mathbb N$ consider the open set $A'_N:=\bigcup_{n=1}^N A_n$ of $\mathcal A(\Omega)$. The sequence  $\{A'_N\}$ is nondecreasing with respect to set inclusion, thus by Lemma~\ref{lem.mon} it is Hausdorff converging in  $\mathcal A(\Omega)$ to its union $\bigcup_{n=1}^\infty A_n$. Taking the liminf as $N\to\infty$ in the inequality above, by the lower semicontinuity of $F$ we obtain the reverse inequality.
\end{proof}

With further assumptions the supremum is indeed a maximum. 
\begin{corollary}[$\sigma$-maxitivity]\label{cor.mon3}
Let $F\colon \mathcal A(\Omega)\to [0,+\infty]$ be a continuous and maxitive set function. 
If $A_n\in\mathcal A(\Omega)$ for every $n\in \mathbb N$ with $A_i\cap A_j=\emptyset$ for every $i,j\in\mathbb N$ such that $i\neq j$ then
\[
F\left(\bigcup_{n=1}^\infty A_n\right)=\max_{n\in\mathbb N} \{F(A_n)\}.
\]
\end{corollary}
\begin{proof}
If $F$ is constant then the result is true. Therefore, we may further assume $F$ to be not constant and proceed by contradiction. Assume that the supremum it is not a maximum, this means that there exists a subsequence (that we do not relabel) $\{A_n\}_n$ such that $\lim_{n\to \infty}F(A_n)=\sup_n F(A_n)$.
For this sequence we have $\cup_n A_n\subset \Omega$ and, since $\Omega$ is bounded and the sets $A_n$ are mutually disjoints, we deduce that $\lim_{n\to \infty}|A_n|=0$.
By Corollary~\ref{cor.bla}, up to a further subsequence (that we still do not relabel), $A_n$ is Hausdorff converging in $\mathcal A(\Omega)$ to an open set $A$ and by lower semicontinuity of the measure  \cite[Proposition 2.2.23]{HP} $|A|\leq \liminf_{n\to\infty} |A_n|$. This combined with the previous limit gives $|A|=0$, namely $A=\emptyset$ ($A$ being open). Therefore, by the continuity of $F$ we obtain $$F(\emptyset)=\lim_{n\to \infty}F(A_n)=\sup_{n\in\mathbb N} F(A_n).$$
Moreover, by Proposition~\ref{prop.mon} $F$ is monotone, thus $F(\emptyset)\leq \inf_{n\in\mathbb N} F(A_n).$ Combining these inequalities for $F(\emptyset)$ we deduce that $F$ is constant along the subsequence $\{A_n\}_n$ so that the supremum is a maximum, a contradiction.
\end{proof}

We are now ready for the last condition which is somehow a local maxitivity property on a disjoint union of a set with a (sufficiently small) ball.

\begin{definition}[Local maxitivity]
A set function $F\colon \mathcal A(\Omega)\to [0,+\infty]$ is \emph{locally maxitive} if for every $A\in\mathcal A(\Omega)\setminus\{\emptyset\}$ there exists $r_A>0$ such that for every $0<r<r_A$ and $x\in\Omega$ 
it holds 
\begin{equation}\label{eq.wm}
F(A_r(x)\cup {B_r^\Omega(x)})=F(A_r(x)),
\end{equation}
where $A_r(x):=A\setminus \overline{B_r(x)}$ and $B_r^\Omega(x):=B_r(x)\cap \Omega$.
\end{definition}

The intersection of the ball $B_r(x)$ with $\Omega$ is needed for $F$ to be well-defined (indeed such ball could be not entirely contained in $\Omega$). However, this condition is actually unnecessary for the set functions that we consider in this paper, as they are defined on all balls in $\mathbb R^2$. 
In the language of maxitive set function the condition \eqref{eq.wm} means that $F$ is maxitive on small balls with maximum value reached by the complement of the ball, as the radius is sufficiently small. The interest on this local maxitivity property, in view of the globlal one \eqref{eq.ma}, is due to the fact that many interesting functionals (such as those depending on the Dirichlet eigenvalues) in fact only verify this local version. To check local maxitivity from maxitivity sometimes it will be useful the following criteria.

\begin{proposition}[Maxitivity\&Local Maxitivity]\label{mimplywm}
Let $F\colon \mathcal A(\Omega)\to [0,+\infty]$ be a lower semicontinuous and maxitive set function, i.e., $F$ satisfies conditions \eqref{eq.ls} and \eqref{eq.ma}. Assume moreover that $F$ is
\begin{itemize}
\item[(i)] \emph{positive on balls}, that is $F(B)>0$ for every ball $B\in\mathcal A(\Omega)$;
\item[(ii)] \emph{continuous on shrinking balls}, that is whenever $\{B^\Omega_n\}_n$ is Hausdorff converging in $\mathcal A(\Omega)$ to $\emptyset$,
\begin{equation}\label{limit}
\lim_{n\to\infty} F(B_n^\Omega)=0,
\end{equation}
where $B_n^\Omega:=B_n\cap\Omega$ and $B_n$ is a ball.
\end{itemize}
Then $F$ is locally maxitive and satisfies \eqref{eq.wm}. 
\end{proposition}
\begin{proof}
Assume by contradiction that $F$ is not locally maxitive. Then there exists $A\in\mathcal A(\Omega)\setminus\{\emptyset\}$ such that for every $n\in\mathbb N$ there exist $0<r_n<1/n$ and $x_n\in\Omega$ 
such that $$F(A_{r_n}(x_n))<F(A_{r_n}(x_n)\cup B^\Omega_{r_n}(x_n))=F(B^\Omega_{r_n}(x_n)),$$
where the last equality holds since $F$ is maxitive. 
If $n\to\infty$ then $r_n\to 0$ and (up to subsequence that we do not relabel) $x_n\to x$ for some $x\in \overline\Omega$, in particular $\{A_{r_n}(x_n)\}_n$ and $\{B^\Omega_{r_n}(x_n)\}_n$ are Hausdorff converging in $\mathcal A(\Omega)$ to $A\setminus\{x\}$ and to $\emptyset$, respectively. Therefore, the lower semicontinuity of $F$ with the continuity of $F$ on shrinking balls yield
\[
0\leq F(A\setminus \{x\})\leq \liminf_{n\to\infty} F(A_{r_n}(x_n))\leq\liminf_{n\to\infty} F(B^\Omega_{r_n}(x_n))=0,
\]
implies $F(A\setminus\{x\})=0$. Now, since $A\setminus \{x\}$ is open and it not the empty set there exists a ball $B\in\mathcal A(\Omega)$ such that $B\subset A\setminus \{x\}$. Then the monotonicity and the strict positivity of $F$ on balls yield $0<F(B)\leq F(A\setminus\{x\})=0$, a contradiction.
\end{proof}

\section{The minimization problem}\label{sec.prob}

We prove the main theorems regarding the minimization problem \eqref{prob}.

\begin{proof}[Proof of Theorem~\ref{teo:ex}]
The existence of a minimizer follows the direct methods of the Calculus of Variations. 
Let
\[
m:=\inf\{ F(\Omega\setminus \Sigma): \Sigma\in\mathcal C(\overline \Omega) \text{ connected  with } \haus(\Sigma)=L\}.
\]
Without loss of generality we may assume $m<+\infty$, otherwise $F$ would be constant and the existence of a minimizer would be trivial. Moreover, notice that $m\geq 0$ thanks to the positivity of $F$. Now, consider a minimizing sequence $\{\Sigma_n\}_n$ in $\mathcal C(\overline \Omega)$ made up of connected sets with $\haus(\Sigma_n)=L$ such that (up to subsequences, that we do not relabel)
\[
\lim_{n\to\infty} F(\Omega \setminus \Sigma_n)= m.
\] 
Again, up to subsequences (that we do not relabel), from Blaschke Theorem (Theorem~\ref{teo.b}) there exists a closed set $\Sigma$ such that $\mathrm d_H(\Sigma_n,\Sigma)\to 0$ as $n\to\infty$. Moreover, Go\l ab theorem~\ref{teo.g} guarantees that $\Sigma$ is connected with $\haus(\Sigma)\leq L$. 
By Lemma~\ref{lemma2} it is possible to extend the set $\Sigma$ to another one $\Sigma'\in \mathcal C(\overline \Omega)$, still connected, and such that $\haus(\Sigma') = L$ (if $\haus(\Sigma)$ were already equal to $L$ just set $\Sigma'=\Sigma$). 
Therefore, $\Omega\setminus \Sigma'\subset \Omega \setminus \Sigma$ and by using the monotonicity of $F$ in \eqref{eq.mo} together with the lower semicontinuity of $F$ in \eqref{eq.ls}, we obtain that
\[
F(\Omega\setminus \Sigma')\leq F(\Omega \setminus \Sigma)\leq \liminf_{n\to\infty} F(\Omega\setminus \Sigma_n)=m,
\]
thus implying the minimality of $\Sigma'$.
\end{proof}

Due to the equality constraint in the minimization problem \eqref{prob}, it is necessary to work with monotone set functions. Without this assumption, according to the Go\l ab theorem, one would only establish the existence of a minimizer within the broader class of continua of $\mathcal C(\overline\Omega)$ with a length (1-dimensional Hausdorff measure) less or equal to $L$.

By employing refined versions of the monotonicity property, those introduced in the previous section, we can establish the Ahlfors regularity.

\begin{proof}[Proof of Theorem~\ref{teo:reg}]
Given a minimizer $\Sigma$ for \eqref{prob} let $$r_0:=\min\{\text{diam($\Sigma$)}/2, r_{\Omega\setminus\Sigma}\}$$ with $r_{\Omega\setminus\Sigma}>0$ the value of local maxitivity of $F$ such that \eqref{eq.wm} holds and $A=\Omega\setminus \Sigma$.
Fix $0<r<r_0$ and $x\in \Sigma$ to denote $B$ the open ball of radius $r$ centered at $x$. The lower bound easily follows from the connectedness of $\Sigma$, see \cite[Theorem~4.4.5]{ambtil}. Indeed, since $0<r<\text{diam($\Sigma$)}/2$ the set $\Sigma\cap \partial B$ is not empty (otherwise $\Sigma$ would not be connected). Therefore, there exists $x\in \Sigma\cap \partial B$ so that $\haus(\Sigma\cap B)\geq r$.

To prove the upper bound, we consider a new continuum $\Sigma'$ just defined by replacing the part of $\Sigma$ inside the ball $B$ (centered at $x$ and of radius $r$) with the boundary of this ball, that is we define
\[
\Sigma':=\left(\Sigma\setminus \overline{B}\right)\cup \partial B.
\] 
Since $0<r<\text{diam($\Sigma$)}$ the set $\partial B\cap \Sigma$ is not empty and the resulting compact set $\Sigma'$ is also connected (indeed each connected components of $\Sigma\setminus \overline{B}$ must intersect $\partial B$, otherwise $\Sigma$ would not be connected).
By using the definition of $\Sigma'$ and of the difference of sets we have
\[
\begin{split}
\Omega\setminus \Sigma'&=\Omega\cap\left(\Sigma\setminus \overline{B}\right)^c\cap(\overline{B}^c\cup B)=\Omega\cap\left(\Sigma^c\cup \overline{B}\right)\cap  (\overline{B}^c\cup B)\\
&=(\Omega\cap \Sigma^c\cap \overline{B}^c)\cup(\Omega\cap \left(\Sigma^c\cup \overline{B}\right)\cap B) =\left((\Omega\setminus \Sigma)\setminus \overline{B}\right)\cup (\Omega\cap B),
\end{split}
\]
and thus
\begin{equation}\label{dim11}
F(\Omega\setminus \Sigma')=F\left(\left((\Omega\setminus \Sigma)\setminus \overline{B}\right)\cup (\Omega\cap B)\right).
\end{equation}
Moreover, since also $r<r_{\Omega\setminus\Sigma}$, then \eqref{eq.wm} holds and we have 
\begin{equation}\label{dim22}
F\left(\left((\Omega\setminus \Sigma)\setminus \overline{B}\right)\cup (\Omega\cap B)\right)=F(\left((\Omega\setminus \Sigma)\setminus \overline{B}\right)).
\end{equation}
Now, we can combine the two equalities \eqref{dim11} and \eqref{dim22} with the monotonicity of $F$ \eqref{eq.sm}, since clearly $(\Omega\setminus \Sigma)\setminus \overline{B}\subset \Omega\setminus \Sigma$, to obtain that   
\[
F(\Omega\setminus \Sigma')\leq F(\Omega\setminus \Sigma).
\]
However, due to the strict monotonicity on enlargements of $F$ in \eqref{eq.sm}, this inequality implies that $\haus(\Sigma')\geq L$, since otherwise thanks to Lemma~\ref{lemma2} we could find an enlargement $\Sigma''$ of $\Sigma'$ with length $L$ that, by the strict monotonicity of $F$ \eqref{eq.sm}, would contradict the minimality of $\Sigma$. 
Therefore,
\[
L\leq \haus(\Sigma')=\haus(\Sigma)-\haus(\Sigma\cap \overline B)+2\pi r= L-\haus(\Sigma\cap \overline B)+2\pi r
\]
which gives the desired density estimate.
\end{proof}

\begin{remark}\label{relax}
The assumptions required on $F$ in Theorems~\ref{teo:ex} and \ref{teo:reg} can of course be relaxed to hold merely on open sets of the form $\Omega\setminus \Sigma$.
\end{remark}

\section{Some applications}\label{sec.app}

We explore a variety of set functions that both motivate and align with all the previously abstract definitions, such as maximum of distances and of torsion functions, linear and nonlinear spectral eigenvalues. By Remark~\ref{relax} it will be sufficient (and sometimes we will do so) to verify the conditions required by Theorems~\ref{teo:ex} and \ref{teo:reg} only for open sets $A=\Omega\setminus \Sigma$ where $\Sigma$ is a continuum of fixed length.

\subsection{Maximum distance problems.} 

We prove the first application to maximum distance problems.
\begin{proof}[Proof of Corollary~\ref{cor.mdp}]
Apply Theorems~\ref{teo:ex} and \ref{teo:reg} to the inradius $R$, which on sets of the form $\Omega\setminus\Sigma\in\mathcal A(\Omega)$ with $\Sigma$ continuum is defined as
$$R(\Omega\setminus \Sigma):=\max_{x\in \overline \Omega} \mathrm{d}(x,\partial \Omega\cup\Sigma).$$
\begin{itemize}
\item[(i)]   
By \eqref{supdist} the Hausdorff convergence of $\partial \Omega\cup\Sigma_n$ to $\partial \Omega\cup\Sigma$ as $n\to\infty$ is equivalent to the uniform convergence of the corresponding distance functions. Since uniform convergence implies pointwise convergence then for $x_0$ such that $R(\Omega\setminus \Sigma)=\mathrm{d}(x_0,\partial \Omega\cup\Sigma)$ we have
\[
R(\Omega\setminus \Sigma)=\mathrm{d}(x_0,\partial \Omega\cup\Sigma)= \lim_{n\to\infty} \mathrm{d}(x_0,\partial \Omega\cup\Sigma_n).
\]
Then, by definition of $R$, $\mathrm{d}(x_0,\partial \Omega\cup\Sigma_n)\leq R(\Omega\setminus\Sigma_n)$ which together with the previous equality provides the lower semicontinuity of $R$.
\item[(ii)] The strict monotonicity \eqref{eq.sm} of $R$ is contained in \cite[Theorem 3.7]{mipast}.
\item[(iii)] The local maxitivity of $R$ follows from Proposition~\ref{mimplywm}. Indeed, the lower semicontinuity has been already proved in (i) and clearly $R$ is a maxitive set function that satisfies \eqref{eq.ma} (the largest ball that can be inscribed in $A_1\cup A_2$ is also the largest one that can be inscribed either in $A_1$ or in $A_2$ and viceversa). Moreover, $R(B_r)=r>0$ for all ball $B_r$ or radius $r>0$ which together with the monotonicity of $F$ yields \eqref{limit}. 
\end{itemize}
The corollary is proved.
\end{proof}

The examination of supremal norms for other functions can unveil new and significant applications. We explore a specific scenario in the next example, focusing on the torsion function.

\subsection{Torsion function.} For an open set $A\in\mathcal A(\Omega)$ consider \emph{the torsion function} $w_A$ maximum of the
\emph{torsional rigidity}
\begin{equation}\label{torsion}
T(A):=\max_{u\in H_0^{1}(A)\atop u\neq 0}\frac{\left(\int_A u(x)\, dx\right)^{2}}{\int_A |\nabla u(x)|^2\, dx},
\end{equation}
normalized to be the unique solution of the Euler-Lagrange equation (also known, in elasticity theory, as the \emph{de Saint-Venant problem})
\[
\begin{cases}
-\Delta w_A=1 \qquad &\text{on $A$}\\
w_A=0 \qquad &\text{on $\partial A$}. 
\end{cases}
\]
It is well-known (see, e.g., \cite[Sections 6.3 and 6.4]{eva}) that $w_A\in C^\infty(A)$ and that, by the strong maximum principle, $w_A$ reaches its minimum at the boundary $\partial A$; $w_A$ is therefore positive inside $A$. Moreover, by the comparison principle, one may compare $w_A$ with $w_B$ for some ball $B$ that contains $A$, so that the maximum of $w_A$ is bounded by the one of $w_B$, i.e. $M(A)\leq  M(B)=R^2/4$, and thus $M(A)$ is finite (in fact it only depends on the diameter of $A$). In this section, we focus on the set function $F$ that is the \emph{maximum of the torsion function} in $A\in\mathcal A(\Omega)$:
\[
M(A)=\max_{x\in A} w_A(x).
\]
Notice that the torsional rigidity \eqref{torsion} corresponds to the $L^1$ norm of $w_A$ on $A$ and it is an additive (thus not maxitive) set function. In particular, the torsional rigidity fails to satisfy the local maxitivity condition \eqref{eq.wm}. The Ahlfors regularity of the minimizers of the torsional rigidity has been established in \cite[Section 4.2]{clls} through a careful analysis of this particular set function.

\begin{corollary}[Maximum of torsion functions]\label{brown}
\!There exists a minimizer of
\[
\min\left\{M(\Omega\setminus \Sigma): \Sigma\in\mathcal C(\overline \Omega) \text{ connected  with } \haus(\Sigma)=L\right\}
\]
and every minimizer is Ahlfors regular.
\end{corollary}
\begin{proof} 
Apply Theorems~\ref{teo:ex} and \ref{teo:reg} to $M$.
\begin{itemize}
\item[(i)] If $\{\Sigma_n\}_n$ Hausdorff converges in $\mathcal C(\overline \Omega)$ to $\Sigma$ then by \cite{sve} (see also \cite[Theorem 3.4.14]{HP}) $w_{\Omega\setminus\Sigma_n}$ converges in $H^1(\Omega)$ to $w_{\Omega\setminus\Sigma}$. The function $w_\Omega-w_{\Omega\setminus\Sigma_n}$ is harmonic in $\Omega\setminus \Sigma_n$ and non-negative on $\partial \Omega\cup \Sigma_n$. By the maximum principle $w_\Omega-w_{\Omega\setminus\Sigma_n}$ must be non-negative in $\Omega\setminus \Sigma_n$, thus $w_{\Omega\setminus \Sigma_n}(x)\leq w_\Omega(x)$ for all $x\in\Omega$. This in particular implies that $w_{\Omega\setminus \Sigma_n}$ is uniformly bounded in $L^\infty(\Omega)$. Up to subsequences (that we do not relabel), $w_{\Omega\setminus\Sigma_n}$ also converges in a weak-* $L^\infty(\Omega)$ sense to $w_{\Omega\setminus\Sigma}$ and the lower semicontinuity of the $L^\infty$-norm w.r.to this convergence yields 
\[
M(\Omega\setminus \Sigma)=\|w_{\Omega\setminus\Sigma}\|_\infty \leq \liminf_{n\to\infty} \|w_{\Omega\setminus\Sigma_n}\|_\infty=M(\Omega\setminus \Sigma_n).
\]
This proves \eqref{eq.ls} for $M$. The lower semicontinuity of $M$ can also be derived from the uniform convergence  of  $w_{\Omega\setminus\Sigma_n}$  to $w_{\Omega\setminus\Sigma}$, see \cite[Proposition~3.6.1]{HP}. 
\item[(ii)]  Let $\Sigma'$ be an enlargement of $\Sigma$. Since $w_{\Omega\setminus\Sigma}-w_{\Omega\setminus\Sigma'}$ is harmonic in $\Omega\setminus \Sigma'$ and non negative on $\partial \Omega\cup \Sigma'$ by the strong maximum principle $w_{\Omega\setminus\Sigma}-w_{\Omega\setminus\Sigma'}$ must be positive in $\Omega\setminus \Sigma'$, namely  $w_{\Omega\setminus \Sigma'}(x)< w_{\Omega\setminus \Sigma}(x)$ for every $x\in\Omega\setminus \Sigma'$. This strict inequality holds in particular for $x'\in\Omega\setminus \Sigma'$ the maximum of $w_{\Omega\setminus\Sigma'}$ and then 
$$M(\Omega\setminus \Sigma')=w_{\Omega\setminus\Sigma}(x')<w_{\Omega\setminus\Sigma}(x')\leq M(\Omega\setminus \Sigma),$$
gives \eqref{eq.sm}. 
\item[(iii)] The local maxitivity of $M$ follows from Proposition~\ref{mimplywm}. Indeed, the lower semicontinuity has been already proved in (i) while the maxitivity of $w_A$ follows from the linearity of the equation which implies that $w_{A_1\cup A_2}=w_{A_1}+w_{A_2}$. Moreover, $M(B_r)=r^2/4>0$ for every ball $B_r$ or radius $r>0$ which together with the monotonicity of $F$ gives \eqref{limit}.
\end{itemize}
The corollary is proved.
\end{proof}

The function $w_{\Omega\setminus \Sigma}$ finds utility in different fields. For instance $w_{\Omega\setminus\Sigma}(x)$ can be interpreted as the \emph{expected first exit time} from $\Omega\setminus \Sigma$ of the \emph{Brownian motion} started at the point $x$ and $M(\Omega\setminus\Sigma)$ is, of course, the maximum of this expected first exit time from $\Omega\setminus \Sigma$. Therefore, in the minimization problem  of Corollary~\ref{brown} one is interested in determining the best location and shape of $\Sigma$ to be placed inside $\Omega$ so as the maximum expected first exit time from $\Omega\setminus\Sigma$ is minimum.

In elasticity theory, the torsion function $w_{\Omega\setminus\Sigma}$ represents the deflection of a membrane $\Omega\setminus \Sigma$ fixed along its boundary $\partial \Omega\cup\Sigma$ and subjected to a uniform load $f=1$. Minimizing the maximal deflection $M(\Omega\setminus\Sigma)$ gives a better and more resilient structure \cite{zuc}. Another classical quantity used in elasticity theory (see \cite{hosste}) is \emph{the maximum shear stress} 
\[
\tau(A)=\max_{x\in A} |\nabla w_A(x)|.
\]
It could be interesting to understand if it fits under this framework (the monotonicity is unclear).

We also mention that $M(\Omega\setminus\Sigma)$ has important applications to efficiency and localization properties of the eigenfunctions, see for instance \cite{bebuka} and references therein.

\subsection{Spectral optimization.}  
One may also push the previous analysis to elliptic partial differential operators with non-constant coefficients: we analyze this topic in this section.
Let $\sigma$ be a $2\times 2$ symmetric matrix-valued measurable function defined over $\Omega$ such that
\begin{equation*}
\sigma_1|\xi|^2\leq \langle \sigma(x)\xi,\xi\rangle\leq \sigma_2|\xi|^2, \quad \text{for all $x\in \Omega$, $\xi\in \mathbb R^2$,}
\end{equation*}
for suitable constants $\sigma_1,\sigma_2>0$ and $V,\rho$ be measurable functions such that $$\rho_1\leq \rho(x) \leq \rho_2\quad  \text{and}\quad 0\leq V(x)\leq V_2 \quad \text{for all $x\in \Omega$,}$$
for suitable constants $\rho_1,\rho_2, V_2>0$ (the positivity of $V$ is just required to deal with positive set functions, but it is not necessary, see \cite[Remark 1.1.3]{hen}).
From these assumptions, given $A\in\mathcal(\Omega)$, it is well-known (see \cite[Section 6.5]{eva} and \cite[Theorem~1.2.2]{hen}) the existence of countably many numbers $\{\lambda_j\}_j$, called \emph{the eigenvalues}, depending on $A$, which provide existence of solutions $\{u_j\}_j$, called \emph{the eigenfunctions}, for the following boundary value problem:
\[
\begin{cases}
-\mathrm{div}(\sigma(x)\nabla u_j(x))+V(x)u_j(x)=\lambda_j \rho(x) u_j(x), \quad &\text{in $A$}\\
u_j=0,\quad  &\text{on $\partial A$}.
\end{cases}
\] 
The variational characterization of these eigenvalues (also known as formula of Courant, Fisher, and Weyl) enables us to express the $j$-th eigenvalue $\lambda_j(A)$ of the open subset $A$ of $\Omega$ as follows:
\begin{equation}\label{eigenvalue}
\lambda_j(A):=\min_{E_j\subset H_0^1(A)}\max_{u\in E_j\atop u\neq 0}\frac{\int_A \langle \sigma(x) \nabla u(x),\nabla u(x)\rangle+V(x)u(x)^2dx}{\int_A \rho(x) u(x)^2dx}
\end{equation}
where the minimum is over all $j$-dimensional subspaces $E_j$ of $H_0^1(A)$. When $\sigma$ is the identity matrix, $\rho=1$, and $V=0$ then we will denote $\lambda_j$ with $\lambda_j^D$ to emphasize that it reduces to the  $j$-th eigenvalue of the Dirichlet-Laplacian. Clearly for every $A\in\mathcal A(\Omega)$ we have that 
\begin{equation}\label{ede}
\frac{\sigma_1}{\rho_2}\lambda_j^D(A)\leq \lambda_j(A)\leq \frac{\sigma_2}{\rho_1}\lambda_j^D(A)+\frac{V_2}{\rho_1}.
\end{equation}
We also recall that the $j$-th eigenvalue of the Dirichlet-Laplacian of a ball of radius $r$ and center $x$ is 
\begin{equation}\label{ball}
\lambda_j^D(B_r(x))=\frac{j_{n,m}^2}{r^2},
\end{equation}
for suitables $n,m\in\mathbb N$, $n\geq 0$, $m>0$, depending on the index $j$, where $j_{n,m}$ is the $m$-th zero of the Bessel function $J_n$. 
\begin{proof}[Proof of Corollary~\ref{spectra}]
Apply Theorems~\ref{teo:ex} and \ref{teo:reg} to $F=f(\lambda_1, \dots, \lambda_k)$ where $k\in\mathbb N$ and $f$ are given.
\begin{itemize}
\item[(i)] The classical \v{S}ver\'ak's $\gamma$-convergence result \cite{sve} (see also \cite[Theorem~3.4.14 and Remark~3.5.2]{HP}) ensures continuity of the eigenvalues for the Hausdorff convergence. This combined with the lower semicontinuity of $f$ yields the lower semicontinuity of $F$ stated in \eqref{eq.ls}.
\item[(ii)] Let $\Sigma'$ be the enlargement of $\Sigma$ given by Lemma~\ref{lemma2}. Each connected component of $\Omega\setminus \Sigma$ has been shrunk by a set of positive one dimensional Hausdorff measure, thus by a set of positive capacity. By strict monotonicity of the eigenvalues with respect to set inclusion $\lambda_j(\Omega\setminus \Sigma')<\lambda_j(\Omega\setminus \Sigma)$ for all $j$ and \eqref{eq.sm} follows (in fact this holds for every enlargement and not only for the one given in Lemma~\ref{lemma2}). This with the assumption that $f$ is strictly monotone yields the strict monotonicity of $F$. 
\item[(iii)] Since in general $\lambda_j$ is not maxitive, Proposition~\ref{mimplywm} cannot be applied to establish the local maxitivity condition. Therefore, we will directly verify that every eigenvalue satisfies the definition of local maxitivity. Fix a set $A\in\mathcal A(\Omega)\setminus\{\emptyset\}$ with a bounded number of connected components of
the complements inside $\Omega$ (by Remark~\ref{relax} it suffices $A=\Omega\setminus \Sigma$ with $\Sigma$ continuum). We first prove a lower bound for $\lambda_1$, then an upper bound on $\lambda_k$, and combine them to find a value $r_A$ for the validity of \eqref{eq.wm}.  
The monotonicity of $\lambda_1$ with the lower bound in \eqref{ede} and the fact in \eqref{ball} that $\lambda_1(B_r(x))=j_{0,1}^2/r^2$, yield 
\begin{equation}\label{lower}
\lambda_1(B_r^\Omega(x))\geq \lambda_1(B_r(x))\geq \frac{\sigma_1}{\rho_2}\lambda_1^D(B_r(x)) =\frac{c_1^2}{r^2},
\end{equation}
where $c_1^2:=\sigma_1j_{0,1}^2/\rho_2$, $r>0$ and $x\in\overline \Omega$ (if $B_r^\Omega(x)$ is empty the inequality is trivial as the left-hand side is to be interpreted $+\infty$). 
On the other hand, let $B_R$ be the largest ball inscribed in $A$ with $R=R(A)$ the inradius of $A$. Since $B_R\subset A$, the monotonicity of $\lambda_k^D$ with \eqref{ball} yields 
\begin{equation}\label{fund}
\lambda^D_k(A)\leq \frac{j_{n,m}^2}{R^2}.
\end{equation}
Moreover, let $x\in\overline\Omega$,
since $\{A_{r}(x)\}_r$ is Hausdorff converging in $\mathcal A(\Omega)$ to $A\setminus\{x\}$ as $r\to 0^+$, the continuity of $\lambda_k^D$ w.r.to the Hausdorff convergence, see \cite[Theorem 3.4.14]{HP}, implies that $\lambda_k^D(A_r(x))$ converges to $\lambda_k^D(A)$ as $r\to0^+$ (in two dimensions the point $x$ has zero capacity \cite[page 102]{HP} so that $A\setminus \{x\}$ and $A$ have same value of $\lambda_k$).
In particular, there exists $r_{2}>0$ such that for every $0<r<r_2$ and $x\in\overline \Omega$ there holds $\lambda_k^D(A_r(x))\leq 2 \lambda_k^D(A)$. From this inequality and the upper bound in \eqref{ede} it follows that 
\begin{equation}\label{dim33}
\lambda_k(A_r(x))\leq \frac{\sigma_2}{\rho_1}\lambda_k^D(A_r(x))+\frac{V_2}{\rho_1}\leq  2\frac{\sigma_2}{\rho_1}\lambda_k^D(A)+\frac{V_2}{\rho_1},
\end{equation}
which combined with \eqref{fund} yields
\begin{equation}\label{upper}
\lambda_k(A_r(x))\leq 2\frac{\sigma_2}{\rho_1}\frac{j_{n,m}^2}{R^2}+\frac{V_2}{\rho_1}:=c_2^2,
\end{equation}
with $0<r<r_2$ and $x\in\overline\Omega$.
Now, let
$$
r_A:=\min\big\{c_1/c_2, r_{2}\big\}.
$$
For every $x\in \overline \Omega$ and $0<r< r_A$ by definition $c_2< c_1/r$ and thus from \eqref{lower} and \eqref{upper} we deduce that $\lambda_k(A_r(x))<\lambda_1(B_r^\Omega(x))$. Since, the eigenvalues of a set are the ordered collection of the eigenvalues of its connected components, for every $1\leq j\leq k$ we have $$\lambda_j(A_r(x)\cup B_r(x))=\lambda_j(A_r(x)).$$ This gives \eqref{eq.wm}, which clearly holds also for $F=f(\lambda_1,\dots, \lambda_k)$. 
\end{itemize}
The corollary is proved. 
\end{proof}

\subsection{Best constants in Poincaré-Sobolev inequalities.} Other eigenvalues, including \emph{nonlinear eigenvalues} in the \emph{homogeneous and superhomogeneous} cases, fit also into our framework. It is more convenient to define them as reciprocals to deal with \emph{optimal Poincaré-Sobolev constants} appearing in the embedding of $W^{1,p}(A)$ into $L^q(A)$ for $A\in\mathcal A(\Omega)$. More precisely, let $p^*$ be the \emph{conjugate exponent} of $p$, that is $p^*:=2p/(2-p)$ if $p<2$ and $p^*:=+\infty$ if $p\geq 2$. For $1<p,q<p^*$ consider the Poincaré-Sobolev constant    
\begin{equation}\label{PS}
C_{p,q}(A):=\max_{u\in W_0^{1,p}(A)\atop u\neq 0}\frac{\left(\int_A u(x)^q\, dx\right)^{p/q}}{\int_A |\nabla u(x)|^p\, dx}.
\end{equation}
The existence of a maximizer is well-known and follows from the direct methods of the Calculus of Variations. This constant is a nonlinear version of the reciprocal first eigenvalue \eqref{eigenvalue} with $k=1$, $\sigma$ be the identity matrix, $\rho=1$, and $V=0$. We prove the lower semicontinuity w.r.to Hausdorff convergence.

\begin{lemma}[Lower semicontinuity of Poincaré-Sobolev constants]\label{ls.PS}
If $\{A_n\}_n$ is Hausdorff converging in $\mathcal A(\Omega)$ to $A$ then
\begin{equation}\label{eq.lsPS}
\liminf_{n\to\infty} C_{p,q}(A_n)\geq C_{p,q}(A).
\end{equation}
\end{lemma}
\begin{proof}
Fix $\epsilon>0$ from \eqref{PS} and the density of $C_c^\infty(A)$ in $H^1_0(A)$ there exists an admissible function $\phi\in C_c^\infty(A)$ for $C_{p,q}(A)$, such that 
\begin{equation}\label{dim555}
C_{p,q}(A)-\epsilon<\frac{\left(\int_A \phi(x)^q\, dx\right)^{p/q}}{\int_A |\nabla \phi(x)|^p\, dx}.
\end{equation}
Now, let $K$ be the support of the function $\phi$. Since $\{A_n\}_n$ is Hausdorff converging in $\mathcal A(\Omega)$ to $A$, then for $n$ large enough $A_n\supset K$, so that $\phi$ becomes also an admissible function for $C_{p,q}(A_n)$ and by using again \eqref{PS} 
\[
\frac{\left(\int_A u(x)^q\, dx\right)^{p/q}}{\int_A |\nabla u(x)|^p\, dx}=\frac{(\int_{A_n} u(x)^q\, dx)^{p/q}}{\int_{A_n} |\nabla u(x)|^p\, dx}\leq C_{p,q}(A_n),
\]
where the first equality follows from the fact that $K$ is contained in $A_n\cap A$.
Combining this inequality with \eqref{dim555}, taking first the liminf as $n\to\infty$ and then the limit as $\epsilon\to 0$, the thesis follows.
\end{proof}

For the rest of this section we need to require $p\leq q$. The main reason for this restriction is that only within this range of indices the Poincaré-Sobolev constant is a maxitive set functions.  This was already stated in \cite[Eq. (49)]{zuc} and proved only for $p=2$ in \cite[Proposition~2.6]{brafra} (by using a \emph{spin formula}, from which the maxitivity condition is not immediate to figure) and for $p=1$ in \cite[Lemma 3.2]{bra}. 

\begin{lemma}[Maxitivity of Poincar\'e-Sobolev constants]\label{max.PS}
Let $1<p\leq q<p^*$. If $A_1,A_2\in\mathcal A(\Omega)$ with $A_1\cap A_2=\emptyset$ then
\[
C_{p,q}(A_1\cup A_2)=\max\{C_{p,q}(A_1),C_{p,q}(A_2)\}.
\]
\end{lemma}
\begin{proof}
Let $u$ be such that 
\[
C_{p,q}(A_1\cup A_2)=\frac{\big(\int_{A_1\cup A_2} u(x)^q\, dx\big)^{p/q}}{\int_{A_1\cup A_2} |\nabla u(x)|^p\, dx}.
\]
By the subadditivity of the power $x^{p/q}$ at the numerator we obtain
\[
\begin{split}
C_{p,q}(A_1\!\cup\! A_2)\!\leq \!
\frac{(\!\int_{A_1} u^q\, dx)^{\frac pq}\!+\!(\!\int_{A_2} u^q\, dx)^{\frac pq}}{\int_{A_1} |\nabla u|^p\, dx\!+\!\int_{A_2} |\nabla u|^p\, dx}\!\leq\! \max_{j=1,2} \frac{(\int_{A_j} u^q\, dx)^{\frac pq}}{\int_{A_j} |\nabla u|^p\, dx}\!\leq\! \max_{j=1,2}C_{p,q}(A_j)
\end{split}
\]
where the last inequality follows from the admissibility of $u$ in the maximization of $C_{p,q}(A_j)$. The reverse inequality is a consequence of the monotonicity of $C_{p,q}$, see Proposition~\ref{prop.mon}. To prove it we test $C_{p,q}(A)$ with $u_j$  the maximizer of $C_{p,q}(A_j)$ for $j=1,2$ to obtain
\[
C_{p,q}(A_1\cup A_2)\geq \frac{\big(\int_{A_j} u_j(x)^q\, dx\big)^{\frac pq}}{\int_{A_j} |\nabla u_j(x)|^p\, dx} = C_{p,q}(A_j).
\]
From the arbitrariness of $j=1,2$ the reverse inequality follows. 
\end{proof}

Let us notice that the Poincaré-Sobolev constants in the \emph{subhomogeneous} regime do not fit into our framework. 
Indeed, when $p>q$ the local maxitivity condition is no longer true. For instance, when $p=2$ and $q=1$, $C_{2,1}$ is the \emph{torsional rigidity} \eqref{torsion}, which as already observed, it is additive (not maxitive) and therefore it fails to satisfy the local maxitivity condition \eqref{eq.wm}.

\begin{corollary}[Best Poincar\'e-Sobolev constants]
Let $1<p\leq q<p^*$ be fixed. The minimization problem 
\begin{equation*}
\min\{C_{p,q}(\Omega\setminus\Sigma): \Sigma\in\mathcal C(\overline \Omega)\text{ connected with } \haus(\Sigma)=L\}
\end{equation*}
has a solution and every minimizer is Ahlfors regular.
\end{corollary}
\begin{proof}
Apply Theorems~\ref{teo:ex} and \ref{teo:reg} to $C_{p,q}$.
\begin{itemize}
\item[(i)] The lower semicontinuity has been provided in Lemma~\ref{ls.PS}.
\item[(ii)] Let $\Sigma'$ be the enlargement of $\Sigma$ given by Lemma~\ref{lemma2}. Each connected component of $\Omega\setminus \Sigma$ has been shrunk by a set of positive one dimensional Hausdorff measure, thus by a set of positive $p$-capacity. By strict monotonicity of the $C_{p,q}$ with respect to set inclusion $C_{p,q}(\Omega\setminus \Sigma')<C_{p,q}(\Omega\setminus \Sigma)$ and \eqref{eq.sm} follows.
\item[(iii)] The local maxitivity of $C_{p,q}$ follows from Proposition~\ref{mimplywm}. Indeed, the lower semicontinuity has been already proved in (i) while the maxitivity of $C_{p,q}$ is provided in Lemma~\ref{max.PS}. Moreover, $C_{p,q}(B_r)>0$ and by homogeneity and translation invariance we have $$C_{p,q}(B_r^\Omega(x))\leq C_{p,q}(B_r(x))= {C_{p,q}((B_1(0))}r^{\frac{2p}{q}+p-2}.$$ The assumption $q<p^*$ is equivalent to ${2p}/{q}+p-2>0$ and then $C_{p,q}$ satisfies \eqref{limit}.
\end{itemize}
The corollary is proved. 
\end{proof}

The case $p=1$ is more nuanced, even in the simpler situation where $q=1$. Indeed, when $q=1$ (and thus also $p=1$), it corresponds to the (reciprocal) \emph{Cheeger constant} (see \cite{kawfri}), for which it remains unclear whether or not the strict monotonicity property \eqref{eq.sm} holds. Due to its notably importance, answering this open question  could be very interesting. 

We conclude by noticing that Poincaré-Sobolev constants serves also as a bridge between more physical quantities and purely geometric ones. Indeed, when $p=q\to+\infty$, the Poincaré-Sobolev constant reduces to the inradius of a set (see \cite{julima} for more information), leading to the class of maximum distance problems \eqref{mdp} studied in Corollary~\ref{cor.mdp}.

\section*{Acknowledgments} 
An anonymous referee who carefully read this paper and suggested several important improvements is kindly acknowledged. The author is partially supported by the PRIN 2022 project 2022R537CS \emph{$NO^3$ - Nodal Optimization, NOnlinear elliptic equations, NOnlocal geometric problems, with a focus on regularity}, founded by the European Union - Next Generation EU; and the GNAMPA 2023 project \emph{Teoria della regolarità per problemi ellittici e parabolici con diffusione anisotropa e pesata}.

\end{document}